\documentclass[11pt]{amsart}
\usepackage[utf8]{inputenc}
\usepackage[a4paper]{geometry}

\usepackage[english]{babel}

\usepackage{amsfonts}
\usepackage{amsthm}
\usepackage{amssymb, amsmath}

\usepackage{hyperref}

\usepackage{xcolor}
\usepackage{cleveref}
\usepackage{tikz}
\usetikzlibrary{calc}

\usepackage{pgfplots}
\pgfplotsset{compat=1.10}
\usepgfplotslibrary{fillbetween}
\usetikzlibrary{patterns}

\usepackage{enumitem}

\theoremstyle{plain}
\newtheorem{theo}{Theorem}[section]
\newtheorem{theoA}{Theorem}

\newtheorem{prop}[theo]{Proposition}
\newtheorem{lemma}[theo]{Lemma}
\newtheorem{lemmaA}[theoA]{Lemma}
\newtheorem{coro}[theo]{Corollary}

\theoremstyle{definition}
\newtheorem{rem}[theo]{Remark}

\newcommand{\C}{\mathbb{C}}
\newcommand{\D}{\mathbb{D}}
\newcommand{\T}{\mathbb{T}}

\newcommand{\Hol}{\mathcal{H}(\mathbb{D})}

\newcommand{\Lim}{\mathcal{L}}
\newcommand{\Conv}{\mathcal{C}}

\newcommand{\Ima}{\operatorname{Im}}
\def\noi{\noindent}

\def\me{\medskip}
\def\sm{\smallskip}
\def\al{\alpha}
\def\ri{\rightarrow}

\title{On the convolution of convex $2$-gons}

\date{\today}

\author[M. Chuaqui]{Martin Chuaqui}
\address{Martin Chuaqui, Facultad de Matemáticas, Pontificia Universidad Católica de Chile, Casilla 306, Santiago 22, Chile}
\email{mchuaqui@mat.uc.cl}

\author[R. Hernández]{Rodrigo Hernández}
\address{Rodrigo Hernández, Facultad de Ingeniería y Ciencias, Universidad Adolfo Ibáñez, Av. Padre Hurtado 750, Viña del Mar, Chile}
\email{rodrigo.hernandez@uai.cl}

\author[A. Llinares]{Adrián Llinares}
\address{Adrián Llinares, Department of Mathematics and Mathematical Statistics, Umeå University, SE-90187 Umeå, Sweden}
\email{adrian.llinares@umu.se}

\author[A. Mas]{Alejandro Mas}
\address{Alejandro Mas, Departamento de Análisis Matemático, Universidad de Valencia, 46100 Burjassot, Spain}
\email{alejandro.mas@uv.es}



\begin{document}

\begin{abstract}
We study the convolution of functions of the form
\[
f_\alpha (z) := \dfrac{\left( \frac{1 + z}{1 - z} \right)^\alpha - 1}{2 \alpha},
\]
which map the open unit disk of the complex plane onto polygons of 2 edges when $\alpha\in(0,1)$. We extend results by Cima by studying limits of convolutions of finitely many $f_\al$ and by considering the convolution of arbitrary unbounded convex mappings. The analysis for the latter is based on the notion of {\it angle at infinity}, which provides an estimate for the growth at infinity and determines whether the convolution is bounded or not. A generalization to an arbitrary number of factors shows that the convolution of $n$ randomly chosen unbounded convex mappings has a probability of $1/n!$ of remaining unbounded. We also extend Cima's analysis on the coefficients of the functions $f_\alpha$ by providing precise asymptotic behavior for all $\alpha$.
\end{abstract}

\maketitle

\section{Introduction}

As usual, let $\D$ and $\T$ be the open unit disk and the unit circle in the complex plane, respectively. If $f$ and $g$ are two holomorphic functions in $\D$, we define their \emph{convolution} as
\[
(f * g)(z) := \dfrac{1}{2 \pi i} \int_{|\zeta| = \frac{1 + |z|}{2}} f(\zeta) g \left( z \zeta^{-1} \right) \, \dfrac{d\zeta}{\zeta}.
\]
Moreover, it can be checked that
\[
 (f * g) (z) = \sum_{n = 0}^\infty a_n b_n z^n,
\]
where $\{ a_n \}_{n \geq 0}$ and $\{ b_n \}_{n \geq 0}$ denote the sequences of Taylor coefficients of $f$ and $g$, so the convolution is also known as \emph{Hadamard product} in the bibliography.

       The convolution of functions satisfying certain geometric properties has been extensively studied (see for instance \cite{MR0674296} or Chapter~14 of \cite{MR0704184} for some general references). For this paper, it is especially relevant the following result, which was originally conjectured by P\'olya and Schoenberg \cite[Appendix~II]{MR0100753} and was proved by Ruscheweyh and Sheil-Small \cite{MR0328051}.

\begin{theoA}[Ruscheweyh, Sheil-Small] \label{thm:RushShei}
Let $f$ and $g$ belong to the class $\Conv$ of univalent functions in $\D$ whose ranges are convex sets. Then, $f * g \in \Conv$ as well.
\end{theoA}

In this paper we will focus on the mappings
\[
f_\alpha (z) := \dfrac{\left( \frac{1 + z}{1 - z} \right)^\alpha - 1}{2 \alpha}, \quad 0 < \alpha < 1.
\]
The mappings $f_\alpha$ are univalent and $f_\alpha(\D)$ corresponds to the convex $2$-gon with vertex at $-\frac{1}{2\alpha}$ and edges with directions $e^{\frac{\alpha \pi}{2}i}$ and $e^{-\frac{\alpha \pi}{2}i}$.

 \noi
If we allow $\alpha$ to assume the value 1, it is clear that
\[
 f_1 (z) = \dfrac{z}{1-z} = \sum_{k = 1}^\infty z^k,
\]
so $f_1$ turns out to be the identity element for the convolution in the subspace of analytic functions vanishing at the origin. On the other hand, taking limits when $\alpha \rightarrow 0$ we get that
\[
f_0 (z) := \lim_{\alpha \rightarrow 0^+} f_\alpha (z) = \dfrac{1}{2} \log \left( \frac{1 + z}{1 - z} \right),
\]
which is again one-to-one and $f_0 (\D) = \{ z \in \C : |\Ima z| < \frac{\pi}{4} \}$.

\me

Much of the motivation of the present paper stems from Cima's paper \cite{MR499120} studying the convolution $f_\alpha * f_\beta$. We will collect several of the results therein and bring attention to his comment on the technical difficulties faced in the analysis. The paper is organized as follows. In Section \ref{sec:Growth2gons}, we extend the results in \cite{MR499120} by giving precise asymptotic behavior for the coefficients of the function $f_\alpha$. We determine when $f_\alpha*f_\beta$ is bounded or not for all choices of the parameters, and study all possible limits of the convolution of finitely many mappings $f_\alpha$.
In Section \ref{sec:GeneralFunct}, we recall the notion of {\it angle at infinity} and use it to establish bounds for the growth of general unbounded convex mappings. From this, we find the probability that the convolution of finitely many randomly chosen unbounded convex mappings is bounded.

\subsection*{Acknowledgements}

We would like to thank I. Efraimidis for his careful reading and helpful comments.

The work of the third author was funded by Grant 275113 of the Research Council of Norway through the Alain Bensoussan Fellowship Programme from ERCIM, and his current research is funded by the postdoctoral scholarship JCK22-0052 from The Kempe Foundations. The fourth author is partially supported by Grant PID2022-136619NB-I00 from Ministerio de Ciencia e Innovación. The third and fourth authors are partially supported by Grant PID2019-106870GB-I00 from Ministerio de Ciencia e Innovación.

Part of this research was conducted during the stay of A.~Llinares and A.~Mas at the Pontificia Universidad Católica de Chile and the Universidad Adolfo Ibáñez. These visits as well as the work of the first two authors were partially funded by Fondecyt Grants 1190380 and 1190756.

\section{Growth of convolutions of 2-gons} \label{sec:Growth2gons}

\subsection{Asymptotic estimates}

For any $\alpha \in (0,1)$, let $\{ g_n (\alpha) \}_{n \geq 0}$ be the sequence of Taylor coefficients of $f_\alpha$. From the power series of $F_\alpha (z) := \left( \frac{1 + z}{1 - z}\right)^\alpha$, it can be seen that
\[
g_n (\alpha) = \dfrac{1}{2 \alpha} \sum_{k = 1}^n 2^k \binom{\alpha}{k} \binom{n - 1}{n - k}, \quad n \geq 2,
\]
where $\binom{a}{b} = \frac{\Gamma(a + 1)}{\Gamma(b + 1) \Gamma (a - b + 1)}$ and $\Gamma$ denotes the Euler Gamma function. However, we have found more useful to follow Cima's approach \cite{MR499120} of treating $f_\alpha$ as the solution of the differential equation
\[
 f_\alpha'(z) = \dfrac{2\alpha f_\alpha (z) + 1}{1 - z^2}, \quad \forall z \in \D,
\]
and $f_\alpha (0) = 0$. Since $g_0(\alpha) = 0$ and $g_1 (\alpha) = 1$, the recursion relation
\begin{equation} \label{eq:Recurrence}
g_{n + 2} (\alpha) = \dfrac{2 \alpha g_{n + 1}(\alpha) + n g_{n}(\alpha)}{n + 2}, \quad \forall n \geq 0,
\end{equation}
can be derived and from it we deduce the following lemma.

\begin{lemmaA}[Cima \cite{MR499120}] \label{LemmaCima}
The sequence $\{ g_n (\alpha) \}_{n \geq 0}$ has the following properties:

\sm
\noi
(i) For every $n \geq 1$, $g_n(\alpha)$ is a polynomial of degree $n-1$ with non-negative coefficients.
In particular, $0 < g_n(\alpha) < 1$ for all $\alpha \in (0,1)$ and $n \geq 2$.

\sm
\noi
(ii) The subsequences $\{g_{2m} (\alpha) \}_{m \geq 1}$ and $\{ g_{2m-1} (\alpha) \}_{m \geq 1}$ are decreasing. If $\alpha \in \left( \frac{1}{2}, 1\right)$, then the chain of inequalities
\[
g_{n} (\alpha) > g_{n + 1} (\alpha) > \dfrac{n}{n + 2 - 2\alpha} g_{n} (\alpha)
\]
holds for every $n \geq 2$.

\sm
\noi
(iii) For every $\alpha \in (0,1)$, there exists a constant $C = C (\alpha)$ such that
\begin{equation} \label{eq:UpperBound}
g_n (\alpha) \leq \dfrac{C}{n^{1-\alpha}}, \quad \forall n \geq 0.
\end{equation}
                   \end{lemmaA}

\noi
Observe that, in particular, this last inequality yields that $f_\alpha * f_\beta$ is bounded if $\alpha + \beta < 1$, as was explicitly mentioned in \cite[Theorem~2]{MR499120}. Our aim is to complement this theorem by considering the case $\alpha + \beta \geq 1$, showing in fact that a lower bound of the same type as in \eqref{eq:UpperBound} is also true.

\begin{theo} \label{TheoConvergence}
The sequence $\{n^{1-\alpha} g_n (\alpha) \}_{n \geq 1}$ is uniformly convergent in intervals of the form $(\alpha_0, 1)$, $0 < \alpha_0 < 1$.
\end{theo}

\begin{proof}
To simplify the notation, we let $G_n (\alpha):= n^{1 - \alpha} g_n (\alpha)$, $n \geq 1$. It follows from \eqref{eq:Recurrence} that
\[
 G_{n + 2} (\alpha) - G_n (\alpha) = \dfrac{2\alpha g_{n+1} (\alpha) + n \big(1 - (1 + 2/n)^\alpha \big) g_n (\alpha)}{(n + 2)^\alpha}.
\]
Since $\alpha \in (0,1)$, we have that
\[
-\dfrac{2\alpha}{n} + \dfrac{2\alpha (1-\alpha)}{n^2} \geq 1 - \left(1 + \dfrac{2}{n} \right)^\alpha \geq -\dfrac{2\alpha}{n}, \quad \forall n \geq 3\]
                    and then
\[
\dfrac{2\alpha \big( g_{n + 1} (\alpha) - g_n (\alpha) \big)}{(n + 2)^\alpha} + \dfrac{2\alpha (1-\alpha) g_n (\alpha)}{n (n+2)^\alpha} \geq G_{n+2} (\alpha) - G_n (\alpha) \geq \dfrac{2\alpha \big( g_{n+1} (\alpha) - g_n (\alpha) \big)}{(n + 2)^\alpha}.
\]
Assume first that $\alpha \geq \frac{1}{2}$. From Lemma \ref{LemmaCima}, we see that
\begin{equation} \label{eq:Aux2}
|g_{n+1} (\alpha) - g_n (\alpha)| = g_n (\alpha) - g_{n+1} (\alpha) < \dfrac{2(1 - \alpha)}{n + 2(1-\alpha)} g_n (\alpha).
\end{equation}
Since $g_k (\alpha)$ is an increasing function of $\alpha$, then $0 < g_k (\alpha) < 1$. Consequently, if $m > j$ we have that
\begin{eqnarray*}
|G_{2m}(\alpha) - G_{2j} (\alpha)| & \leq & \sum_{k = j}^{m - 1} |G_{2k+2} (\alpha) - G_{2k} (\alpha)| \\
& \leq & \sum_{k = j}^m \left( \dfrac{2\alpha | g_{2k+1} (\alpha) - g_{2k} (\alpha) |}{(2k+2)^\alpha} + \dfrac{2\alpha(1 - \alpha) g_{2k} (\alpha)}{2k (2k + 2)^\alpha} \right) \\
& \leq & \alpha(1-\alpha) 2^{1 - \alpha} \sum_{k = j}^m \left( \dfrac{1}{\big(k + 1 - \alpha \big)(k+1)^\alpha} + \dfrac{1}{k (k + 1)^\alpha} \right) \\
& \leq & \sum_{k = j}^m \left(\dfrac{1}{(2k+1)\sqrt{k+1}} + \dfrac{1}{2k \sqrt{k+1}} \right),
\end{eqnarray*}
and hence $\{ G_{2m} (\alpha) \}_{m \geq 1}$ is Cauchy in $m$ uniformly in $\al$, and is therefore uniformly convergent. The same analysis applies to $\{ G_{2m + 1} (\alpha) \}_{m \geq 0}$, and we will show that both limits coincide. Indeed, Lemma \ref{LemmaCima} yields that
\[
G_{2m} (\alpha) \geq \left( \dfrac{2m}{2m+1} \right)^{1 - \alpha} G_{2m+1} (\alpha) \geq \dfrac{m}{m + 1 - \alpha} G_{2m} (\alpha), \quad \forall m \geq 1,
\]
and therefore $\{G_n (\alpha)\}_{n \geq 1}$ is uniformly convergent in $\left( \frac{1}{2}, 1 \right)$.

Now, assume that $0 < \alpha < \frac{1}{2}$. Since $2\alpha \big( g_{n+1} (\alpha) - g_n (\alpha) \big)$, $n \geq 2$, corresponds to the $(n+1)$-th coefficient of $H_\alpha (z):=(1 - z) F_\alpha (z)$, and since
\[
H_\alpha (z) = (1 + z)^\alpha (1 - z)^{1-\alpha} = H_{1-\alpha} (-z),
\]
we see from \eqref{eq:Aux2} and the estimate $g_n (1 - \alpha) = O ( n^{-\alpha} )$ that
\[
|g_{n+1}(\alpha) - g_n (\alpha)| = \dfrac{1-\alpha}{\alpha} |g_{n+1} (1 - \alpha) - g_n (1 - \alpha)| = O \left( \dfrac{1}{n^{1 + \alpha}} \right).
\]
Hence again $\{ G_{2m} (\alpha) \}_{m \geq 1}$ and $\{ G_{2m + 1} (\alpha)\}_{m \geq 0}$ are uniformly convergent on the intervals $\left (\alpha_0, \frac{1}{2} \right]$, $0 < \alpha_0 < \frac{1}{2}$. Finally, both limits must coincide because
\[
|G_{n+1} (\alpha) - G_n (\alpha)| \leq (n+1)^{1-\alpha} \left( \dfrac{2(1-\alpha)}{n + 2\alpha} + 1 - \left( \dfrac{n}{n+1} \right)^{1-\alpha} \right) = O \left( \dfrac{1}{n^\alpha} \right).
\]

\end{proof}

Note that the proof provides additional qualitative information about the limit $L(\alpha)=\lim_{n \rightarrow \infty} G_n (\alpha)$ as a function of $\alpha$. This function must be continuous on $(0,1)$ and it is not difficult to see that it must be strictly positive with $|1 - L(\alpha)| = O(1 - \alpha)$ when $\alpha \rightarrow 1^-$. The following theorem states this more precisely.

\begin{theo} \label{thm:LimitCoeff}
For every $0 < \alpha < 1$, we have that
\begin{equation} \label{eq:Limit}
L(\alpha) = \dfrac{2^{\alpha - 1}}{\Gamma(\alpha + 1)}.
\end{equation}
\end{theo}

\begin{proof}
On the one hand, we trivially have
\[
\lim_{r \rightarrow 1^-} (1 - r)^\alpha f_\alpha (r) = \dfrac{2^{\alpha-1}}{\alpha}.
\]
On the other hand, observe that $c_n (\alpha) := \frac{\Gamma(n + \alpha)}{\Gamma (\alpha) n!}$ is the $n$-th Taylor coefficient of $(1 - z)^{-\alpha}$, $\alpha > 0$. The Euler-Gauss formula for the Gamma function (see for instance \cite[p.~255]{MR2933135}) implies that
\[
\lim_{n \rightarrow \infty} \dfrac{g_n (\alpha)}{c_n (\alpha)} = \Gamma (\alpha) L(\alpha).
\]
Let $\varepsilon > 0$ and $n_0$ be such that
\[
\Gamma (\alpha) L(\alpha) - \varepsilon \leq \dfrac{g_n (\alpha)}{c_n (\alpha)} \leq \Gamma (\alpha) L(\alpha) + \varepsilon, \quad \forall n \geq n_0.
\]
Then
\[
\lim_{r \rightarrow 1^-} (1 - r)^\alpha f_\alpha (r) = \lim_{r \rightarrow 1^-} (1 - r)^\alpha \sum_{n = n_0}^\infty g_n (\alpha) r^n \in [\Gamma(\alpha) L(\alpha) - \varepsilon, \Gamma(\alpha) L(\alpha) + \varepsilon],
\]
and by letting $\varepsilon\rightarrow 0$ we find that
\[
\Gamma(\alpha) L(\alpha) = \dfrac{2^{\alpha - 1}}{\alpha},
\]
which implies \eqref{eq:Limit}.
\end{proof}

\begin{rem}
Observe that,  $L(1) = 1$ because $g_n (1) = 1$ for all $n$. However, because $\lim_{\alpha \rightarrow 0^+} L( \alpha ) = \frac{1}{2}$, $L$ is only defined on $(0,1]$.
\end{rem}

Following a similar analysis as in the proof of Theorem \ref{thm:LimitCoeff} along with Carath\'eodory theorem \cite[Theorem~2.1]{MR1217706} (the unbounded version of which can apparently be attributed to Torhorst \cite{MR1544451} as well), we can finally state an extension of Theorem 2 in \cite{MR499120}.

\begin{theo} \label{thm:FiniteConv}
For $\alpha, \beta \in (0,1]$ the following hold:

\sm
\noi
(i) If $\alpha + \beta < 1$, then $f_\alpha * f_\beta$ is bounded.

\sm
\noi
(ii) If $\alpha + \beta = 1$, then
\[
\lim_{z \rightarrow 1} \dfrac{(f_\alpha * f_\beta) (z)}{-\log (1 - |z|)} = \dfrac{1}{2 \Gamma (\alpha + 1) \Gamma (\beta + 1)}.
\]

\sm
\noi
(ii) If $\alpha + \beta > 1$, then
 \[
\lim_{z \rightarrow 1} (1 - |z|)^{\alpha + \beta - 1} (f_\alpha * f_\beta) (z) = 2^{\alpha + \beta - 2} \dfrac{\Gamma (\alpha + \beta - 1)}{\Gamma(\alpha + 1) \Gamma (\beta + 1)}.
\]
\end{theo}

\subsection{Arbitrary number of convolutions}

In this section we study the convolution of an arbitrary number of 2-gons, including the limiting case $f_0$. For $\alpha \in \D \setminus \{ 0 \}$, we define the function $f_\alpha$ as the rotation of $f_{|\alpha|}$ by $e^{i\arg \alpha}$, {\it i.e.},
\[
f_\alpha (z) := e^{-i \arg \alpha} f_{|\alpha|} \big( e^{i \arg \alpha} z \big), \quad z \in \D.
\]
We are interested in the set of all possible limits of finite convolutions of 2-gons given by
\[
\Lim := \big \{ f \: : \: f = \lim_{n \rightarrow \infty} f_{\alpha_1} * \ldots * f_{\alpha_n}, \{ \alpha_n \}_{n \geq 1} \subset \D \setminus \{ 0 \} \big \} \subset \Conv.
\]

According to the lemma below, we may assume that all of the elements of the sequence $\{ \alpha_n \}_{n \geq 1}$, except perhaps $\al_1$, are positive.

\begin{lemma}
\[
\Lim = \big \{ f \: : \: f = \lim_{n \rightarrow \infty} f_{\alpha_1} * \ldots * f_{\alpha_n}, \alpha_1 \in \D \setminus \{ 0 \}, \{ \alpha_n \}_{n \geq 2} \subset (0,1) \big \}.
\]
\end{lemma}

\begin{proof}

Let $\{ \alpha_n \}_{n \geq 1} \subset \D \setminus \{ 0 \}$ and consider the function $f \in \Hol$ such that
\[
f_{\alpha_1} * \ldots * f_{\alpha_n} \rightarrow f
\]
uniformly for all compact subset $K$ of $\D$. Since $g_2(\alpha) = \alpha$, we have that the sequence of partial products $\left \{ \prod_{j = 1}^n \alpha_j \right \}_{n \geq 1}$ is convergent. Let $p$ be its limit. At this point, we split the proof by cases.
\sm
\noi
(i) If $p = 0$, we are going to show that $f(z) = z$. Without loss of generality, we might assume that $|\alpha_n| \geq \frac{1}{2}$ for all $n \geq 1$. Because of Lemma \ref{LemmaCima}, we have that
\[
\prod_{j = 1}^n |\alpha_j| \geq \prod_{j = 1}^n \left | g_k (|\alpha_j|) \right |, \quad n \geq 1, k \geq 2,
\]
and thus it is clear that $f(z) = z$, which is obviously the limit of ``convolutioning'' $f_\frac{1}{2}$, with itself infinitely many times.

\sm
\noi
(ii) On the other hand, if $p \neq 0$ we have that the infinite product converges. By continuity, we have that $\prod_{n \geq 1} |\alpha_n| = |p|$, and therefore the product $\lambda := \prod_{n \geq 1} e^{i \arg \alpha_n}$ is also convergent. Observe that $|\lambda| = 1$, and consider the sequence $\beta_1 := \lambda |\alpha_1|$ and $\beta_n := |\alpha_n|$, $n \geq 2$. We are going to show that
\begin{equation} \label{LemmaAux}
f = \lim_{n \rightarrow \infty} f_{\beta_1} * \ldots * f_{\beta_n}.
\end{equation}
Indeed, observe that the $k$-th Taylor coefficient of $f$ is given by
\[
\lim_{n \rightarrow \infty} \prod_{j = 1}^n g_{k} (|\alpha_j|) e^{i (k-1) \arg \alpha_j} = \lambda^{k-1} \lim_{n \rightarrow \infty} \prod_{j = 1}^n g_k(|\alpha_j|),
\]
which is trivially equal to the $k$-th coefficient of $\lim_{n \rightarrow \infty} f_{\beta_1} * \ldots * f_{\beta_n}$ because of the definition of the sequence $\{ \beta_n \}_{n \geq 1}$. Therefore, it is clear that the identity \eqref{LemmaAux} holds.
\end{proof}

It may seem from this proof that $\{ \alpha_n \}_{n \geq 1}$ and $\{ \beta_n \}_{n \geq 1}$ can give rise to the same limit as long as
\[
\lim_{n \rightarrow \infty} \prod_{j = 1}^n \alpha_j = \lim_{n \rightarrow \infty} \prod_{j = 1}^n \beta_j,
\]
but this is not true. For instance, if $\{ \alpha_n \}_{n \geq 1} \subset (0,1)$ is a sequence such that its infinite product is convergent, and consider the sequence $\{ \beta_n \}_{n \geq 1}$ defined as
\[
\beta_{2m - 1} = \beta_{2m} = \sqrt{\alpha_m}, \quad \forall m \geq 1.
\]
A straightforward computation shows that the function $g_3 (\alpha) = 1 - \frac{2}{3} (1 - \alpha^2)$ is a \emph{supermultiplicative} function. In other words,
\[
g_3 (\alpha \beta) > g_3 (\alpha) g_3 (\beta)
\]
for all $\alpha, \beta \in (0,1)$. Then, we have that
\[
\prod_{j = 1}^n g_3(\alpha_j) > \prod_{j = 1}^{2n} g_3 (\beta_j), \quad \forall n \geq 1,
\]
and hence the convergence of the product $\prod_{j \geq 1} \beta_j$ yields that
\[
\prod_{j = 1}^\infty g_3 (\alpha_j) > \prod_{j = 1}^\infty g_3 (\beta_j) > 0,
\]
which suffices to see that
\[
 \lim_{n \rightarrow \infty} f_{\alpha_1} * \ldots * f_{\alpha_n} \neq \lim_{n \rightarrow \infty} f_{\beta_1} * \ldots * f_{\beta_n}.
\]

The following result highlights the effects of having a vanishing Taylor coefficient in $\Lim$.

\begin{prop} \label{ZeroCoeffLim}
If $f \in \Lim$ and $f^{(k_0)} (0) = 0$ for some $k_0 \geq 2$, then $f(z) \equiv z$.
\end{prop}
\begin{proof}

Let $\alpha_1 \in \D \setminus \{ 0 \}$ and $\{ \alpha_n \}_{n \geq 2} \subset (0,1)$ be such that $f = \lim_{n \rightarrow \infty} f_{\alpha_1} * \ldots * f_{\alpha_n}$. Then, the condition $f^{(k_0)} (0) = 0$ implies that
\[
\lim_{n \rightarrow \infty} \prod_{j = 2}^n  g_{k_0} (\alpha_j) = 0,
\]
since $g_k(|\alpha_1|) > 0$. Without loss of generality, assume that $f^{(k_0 - 1)} (0) \neq 0$.

Two cases are possible:

\sm
\noi (i) If there exists a subsequence $\{ \alpha_{n_j} \}_{j \geq 1}$ such that $\alpha_{n_j} \leq \frac{1}{2}$ for all $j$, then the monotonocity of the polynomials $\{ g_k (\alpha) \}_{k \geq 2}$ with respect to $\alpha$ implies that
\[
\lim_{n \rightarrow \infty} \prod_{j = 1}^n g_k (\alpha_j) = 0, \quad \forall k \geq 2,
\]
and hence $f$ must be the identity.

\sm
\noi
(ii) If there exists $N \geq 2$ such that $\alpha_n > \frac{1}{2}$ for all $n \geq N$, we will show that
\[
\lim_{n \rightarrow \infty} \prod_{j = N}^n \alpha_j = 0,
\]
and then Lemma \ref{LemmaCima} yields that $f(z) \equiv z$.

In order to get a contradiction, assume that $\prod_{n \geq 1} \alpha_n$ is convergent. In other words, that $k_0 > 2$. Then, the product $\prod_{n \geq 1} \frac{k_0 - 1}{k_0 + 1 - 2\alpha_n}$ is also convergent. However, Lemma \ref{LemmaCima} implies that
\[
\prod_{j = N}^n g_{k_0} (\alpha_j) > \left( \prod_{j = N}^n \frac{k_0 - 1}{k_0 + 1 - 2\alpha_n} \right) \left( \prod_{j = N}^n g_{k_0 - 1} (\alpha_j) \right),
\]
so the product on the left-hand side is bounded from below by the multiplication of two convergent and positive products, which is not possible. Thus, the partial products of the sequence
$\{ \alpha_n \}_{n \geq 1}$ converge to zero, which means that $f$ is the identity.
\end{proof}

The following analogue of Theorem \ref{thm:FiniteConv} can also be drawn.

\begin{theo} \label{thm:ConvGrowth}
Let $\{ \alpha_n \}_{n \geq 1} \subset (0,1]$ and $B := \sum_{n \geq 1} (1 - \alpha_n) \in (0, \infty]$. Consider the function $f := \lim_{n \rightarrow \infty} f_{\alpha_1} * \ldots * f_{\alpha_n}$. Then, we have the following:

\sm
\noi
(i) If $B > 1$, then $f$ is bounded.

\sm
\noi
(ii) If $B = 1$, then
\begin{equation} \label{eq:Asympt1}
\lim_{z \rightarrow 1} \dfrac{f(z)}{-\log(1-|z|)} = \dfrac{1}{2} \prod_{n = 1}^\infty \dfrac{1}{\Gamma(\alpha_n + 1)}.
\end{equation}

\sm
\noi
(iii) If $B < 1$, then
\[
\lim_{z \rightarrow 1} (1 - |z|)^{1 - B} f(z) = \dfrac{\Gamma \left( 1 - B \right)}{2^B} \prod_{n = 1}^\infty \dfrac{1}{\Gamma (\alpha_n + 1)}.
\]
\end{theo}

\begin{proof}

\sm
\noi
(i) If $B > 1$, then there exists $N$ big enough such that the partial sum $\sum_{n = 1}^N (1 - \alpha_n)$ is greater than $1$ as well. Thus, we have that
\[
 \| f \|_{\infty} \leq \| f_{\alpha_1} * \ldots * f_{\alpha_N} \|_{\infty} < \infty.
\]

\sm
\noi
(ii) Since the sum is convergent observe $\alpha_n < \frac{1}{2}$ for, at most, a finite number of indexes. Because $|1 - L(\alpha)| = O (1 - \alpha)$ in $\left( \frac{1}{2}, 0 \right)$, note that, in particular, we have that the product $\prod_{n \geq 1} L (\alpha_n)$ is convergent. As a consequence of the uniform convergence of $\{ G_k (\alpha) \}_{k \geq 1}$, we have that there exists $k_0 \geq 0$ such that
\[
\prod_{n = 1}^\infty L (\alpha_n) - \varepsilon \leq k \prod_{n = 1}^\infty g_k (\alpha_n) \leq \prod_{n = 1}^\infty L (\alpha_n) + \varepsilon, \quad \forall k \geq k_0,
\]
and therefore \eqref{eq:Asympt1} holds.

\sm
\noi
(iii) It is similar to the previous case.
\end{proof}

Although finite convolutions of elements of $\{ f_\alpha \}_{0 \leq \alpha \leq 1}$ do not belong to $\Lim$, the latter theorem can be also applied for this case. Nevertheless, these finite convolutions may be approximated by functions in $\Lim$, as it is implicitly shown in the following proposition.

\begin{prop}
The functions $f_0$ and $f_1$ do not belong to $\Lim$, but they belong to the closure of $\Lim$. In particular, $\Lim$ is not closed.
\end{prop}

\begin{proof}
Since the coefficients of $f_\alpha$, $\alpha \in \D$, have moduli strictly less than $1$, it is clear that $f_1 \not \in \Lim$. On the other hand, $f_0$ does not belong to $\Lim$ because of Proposition \ref{ZeroCoeffLim}.

Now, consider the quantities
\[
\alpha_{n,j} := \left ( 1 - \dfrac{1}{j + 1} \right)^{2^{-n}}, \quad n, j \geq 1,
\]
and the sequence of functions $F_{j} := \lim_{n \rightarrow \infty} f_{\alpha_{1,j}} * \ldots * f_{\alpha_{n,j}}$. It is clear that $\{ F_j \}_{j \geq 1} \subset \Lim$.

Observe that $g_k (\alpha_{n,j}) \leq g_k (\alpha_{n, j+1}) < 1$ for all $n,j \geq 1$ and $k \geq 2$, so the sequence of $k$-th Taylor coefficients of the functions $\{ F_j \}_{j \geq 1}$ is convergent. Hence, the limit $F = \lim_{j \rightarrow \infty} F_j$ exists. It remains to show that $F \equiv f_1$, but this a consequence of the extremal properties of $f_1$. Indeed, it is a well-known fact that
\[
|a_n| \leq 1, \quad \forall n \geq 2
\]
if $f(z) = z + \sum_{n \geq 2} a_n z^n \in \Conv$, and that equality is attained if and only if $f(z) \equiv \overline{\lambda} f_1 (\lambda z)$ for some unimodular constant $\lambda$. It is clear that $F \in \Conv$ and, by the construction of $\{ F_j \}_{j \geq 1}$, its second Taylor coefficient is equal to 1. Therefore $F$ must coincide with $f_1$.

Once we know that $f_1 \in \overline{\Lim}$, it is easy to see that $f_0 \in \overline{\Lim}$ (for example, it can be checked that $G_j := f_{1/j} * F_j \in \Lim$ converges to $f_0$ when $j \rightarrow \infty$) and consequently we have finished the proof.
\end{proof}

Unfortunately, a complete characterization of $\overline{\Lim}$ seems to be out of reach. However, we can provide the following necessary condition.

\begin{prop}
Let $f \in \overline{\Lim}$. Then, one of the following statements is true:

\sm
\noi
(i) All of the Taylor coefficients of $f$ are different than $0$.

\sm
\noi
(ii) The function $f$ is an odd function and all the odd Taylor coefficients are different than 0.

\sm
\noi
(iii) $f(z) \equiv z$.
\end{prop}

\begin{proof}
Take $f \in \overline{\Lim}$ different than the identity. Consider a sequence of functions $\{ F_j \}_{j \geq 1} \subset \Lim$ such that $F~=~\lim_{j \rightarrow \infty} F_j$. Since $f(z) \not \equiv z$, we might assume that $F_j(z) \not \equiv z$ for all $j$. Let $\{ \alpha_{n, j} \}_{n \geq 1}$ (all of them positive except, possibly, the first one) such that $F_j~=~\lim_{n \rightarrow \infty} f_{\alpha_{1,j}} * \ldots * f_{\alpha_{n,j}}$. Note that we have that
\begin{equation} \label{Even1}
g_k (|\alpha_{1, j}|) \prod_{n = 2}^\infty g_k (\alpha_{n,j}) \geq |\alpha_{1, j}|^{k-1} \prod_{n = 2}^\infty \alpha_{n,j}^{k - 1}, \quad \forall k \geq 2, \forall j \geq 1,
\end{equation}
since the elementary estimate $g_k(\alpha) \geq \alpha^{k - 1}$ holds for every $0 < \alpha < 1$ and $k~\geq~2$. Moreover, Lemma \ref{LemmaCima} implies that
\begin{equation} \label{Even2}
|\alpha_{1,j}| \prod_{n = 2}^\infty \alpha_{n,j} \geq g_{2m} (|\alpha_{1,j}|) \prod_{n = 2}^\infty g_{2m} (\alpha_{n,j})
\end{equation}
and
\begin{equation*}
g_3 (|\alpha_{1, j}|) \prod_{n = 2}^\infty g_3 (\alpha_{n,j}) \geq g_{2m + 1} (|\alpha_{1,j}|)\prod_{n = 2}^\infty g_{2m + 1} (\alpha_{n,j})
\end{equation*}
for all $j, m \geq 1$. Consequently, there is no odd number $k$ such that
\[
\lim_{j \rightarrow \infty} g_{k} (|\alpha_{1,j}|) \prod_{n = 2}^\infty g_{k} (\alpha_{n,j}) = 0
\] because $f(z) \not \equiv z$. If $k$ is an even number such that the $k$-th Taylor coefficient of $f$ is zero, it follows from \eqref{Even1} and \eqref{Even2} that $f$ must be an odd function.
\end{proof}

\section{Estimates for unbounded convex sets} \label{sec:GeneralFunct}

In this section we deduce estimates for general unbounded convex sets. We recall that the angle at infinity of a convex set is the infimum of the angles of all sectors containing that set. It follows from convexity that an unbounded convex set contains a $2$-gon of this amplitude if it is positive.

\begin{theo}
Let $0 < \alpha < 1$. Let $g$ be a unbounded convex mapping with $g(1) = \infty$. If $g(\D)$ has an angle at infinity of $\pi \alpha$, then there exists $\delta > 0$ such that
\begin{equation} \label{eqn:LowerBound}
|1 - z|^\alpha |g (z)| \geq \delta
\end{equation}
when $z \rightarrow 1$.
\end{theo}

\begin{proof}

Without loss of generality, we may assume that $\Omega = g(\D)$ contains a sector of the form $S_\alpha := c + f_\alpha (\D)$, that $g(0) = c + f_\alpha(0)$ and that $0$ lies outside the closure $\overline{g(\D)}$. Thus, there exists an analytic function $\varphi : \D \rightarrow \D$ with $\varphi(0) = 0$ such that $c + f_\alpha(z) = g \big( \varphi(z)\big)$ for all $z$. Since $g(1) = \infty$, then we necessarily have that $\varphi(1) = 1$ and the Julia-Wolff lemma \cite[Section~4.1]{MR1217706} yields that the angular derivative
\[
\varphi'(1) := \lim_{z \rightarrow 1, z \in \Delta} \dfrac{\varphi(z)-1}{z-1},
\]
where $\Delta$ is a Stolz angle at $1$, exists and $|\varphi'(1)| \in [1, \infty]$. If $z = \varphi(\zeta)$ for some $\zeta \in \D$, we have that
\[
|1 - z|^\alpha |g(z)| = |1 - \varphi(\zeta)|^\alpha |c + f_\alpha (\zeta)|
\]
so
\[
\lim_{z \rightarrow 1, z \in \Delta \cap \varphi(\D)} |1-z|^\alpha |g(z)| = \dfrac{2^{\alpha-1}|\varphi'(1)|^\alpha}{\alpha} > 0.
\]

It remains to analyze the behavior as $z \rightarrow 1$ with $z \in \D \setminus D_\alpha$, where $D_\alpha := \varphi (\D)$. Since $g$ is convex, then its Schwarzian norm $\| Sg \|$ is less than or equal to 2. Because the angle at infinity is positive it follows also that $g(\D)$ is a quasidisk on the Riemann sphere, hence $\| Sg \|<2$ by Theorem 4 in \cite{MR2865506}. We conclude that the function
\[
G = \frac{1}{g} + 1
\]
maps $\D$ onto a bounded quasidisk, in particular, $G(\D)$ is a John domain. If $B(z)$ denotes the closed Carleson box $\left \{\zeta \, : \, |z| \leq |\zeta| \leq 1, \left |\arg \left(\overline{\zeta}z\right) \right| \leq \pi (1-|z|) \right\}$, then there exists a uniform $C > 0$ such that
\[
\operatorname{diam} G \big( B(z) \big) \leq C d \big( G(z), \partial G(\D) \big).
\]
Writing $z = re^{i\theta}$, this estimate implies that
\[
\dfrac{1}{|g(e^{i\theta})|} - \dfrac{1}{|g(z)|} \leq |G(z) - G(e^{i\theta})| \leq C |G(z) - G(1)| = \dfrac{C}{|g(z)|}
\]
and therefore we have that $|g(e^{i\theta})| \geq \frac{|g(z)|}{1 + C}$. Since $|1-z| \sim |1 - e^{i\theta}|$, we conclude that there is a $\delta > 0$ such that
\[
 0 < \delta \leq |1-z|^\alpha |g(z)| \lesssim |1 - e^{i\theta}|^\alpha |g(e^{i\theta})|, \quad z \rightarrow 1, z \in D_\alpha.
\]
A similar argument replacing $e^{i\theta}$ by $\zeta \in \D \setminus D_\alpha$ shows that
\[
 \delta \leq |1 - \zeta|^\alpha |g(\zeta)|, \quad \zeta \rightarrow 1^-,
\]
and we are done.
\end{proof}

On the other hand, it follows from subordination that if the angle at infinity of the unbounded convex region $\Omega=g(\D)$ is $\pi\al$, then for any $\beta > \al$ there exists $K = K(\beta)$ such that
\begin{equation} \label{eqn:UpperBound}
|1-z|^\beta |g(z)| \leq K, \quad z \rightarrow 1.
\end{equation}

An interesting issue is whether there exists a sector of minimal amplitude containing $\Omega$. This will happen only if the vertices of the sectors of a minimizing sequence remain in a bounded region of the plane. For example, a parabolic region has zero angle at infinity but no parallel strip contains it. Simple modifications provide similar examples for positive angle at infinity (for example when $\Omega$ is the intersection of angles of decreasing amplitudes with aligned and divergent vertexes, like the one represented in Figure \ref{fig:VertexInf}), so in principle the choice $\beta = \alpha$ is not allowed.

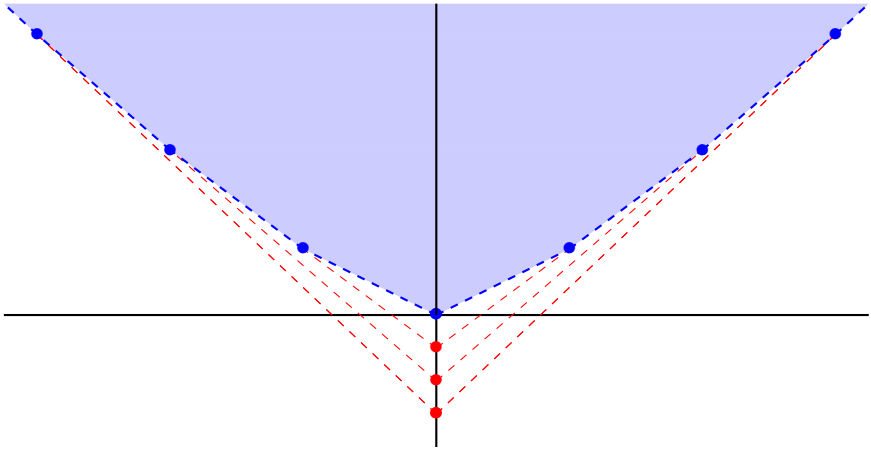
\begin{figure}[h]
\centering
\begin{tikzpicture}[x=1.75cm, y=1.75cm]
%%%%%
% Parameters
%%%%%
\def \maxlevel {3};
\pgfmathsetmacro{\valu}{0};

%%%%%
% Drawings
%%%%%

%Drawing the axis
\draw[black, thick] (-\maxlevel - 1/4,0) -- (\maxlevel + 1/4,0);
\draw[black, thick] (0,-\maxlevel/4-1/4) -- (0,0);

%Drawing the set
\foreach \i in {1,...,\maxlevel}
{
\pgfmathsetmacro{\vertex}{-\i/4};
\pgfmathsetmacro{\slope}{1-2^(-\i)};
\pgfmathsetmacro{\secondvalue}{\valu + \slope};

%Drawing an endpoint
\draw[blue] (\i-1,\valu) node{$\bullet$};
\draw[blue] (1-\i,\valu) node{$\bullet$};

%Drawing the vertex
\draw[red] (0,\vertex) node{$\bullet$};
\draw[red, dashed] (0,\vertex) -- (\i, \secondvalue);
\draw[red, dashed] (0,\vertex) -- (-\i, \secondvalue);

%Drawing the segment
\draw[blue, dashed, thick] (\i - 1, \valu) -- (\i, \secondvalue);
\draw[blue, dashed, thick] (-\i + 1, \valu) -- (-\i, \secondvalue);

%Filling the set
\fill[blue, opacity=0.2] (\i - 1, \valu) -- (\i, \secondvalue) -- (-\i, \secondvalue) -- (-\i + 1, \valu) --cycle;

%Updating the value
\global\let\valu = \secondvalue;
}

%Drawing the last vertex
\draw[red] (0,-\maxlevel/4) node{$\bullet$};
\draw[red, dashed] (0, -\maxlevel/4) -- (\maxlevel, \valu);
\draw[red, dashed] (0,-\maxlevel/4) -- (-\maxlevel, \valu);
\draw[blue] (\maxlevel, \valu) node{$\bullet$};
\draw[blue] (-\maxlevel, \valu) node{$\bullet$};

%Drawing the last segment
\pgfmathsetmacro{\slope}{1-2^(-\maxlevel-1)};
\draw[blue, dashed, thick] (\maxlevel, \valu) -- (\maxlevel + 1/4, \valu + \slope/4);
\draw[blue, dashed, thick] (-\maxlevel, \valu) -- (-\maxlevel - 1/4, \valu + \slope/4);
\fill[blue, opacity=0.2] (\maxlevel, \valu) -- (\maxlevel + 1/4, \valu + \slope/4) -- (-\maxlevel - 1/4, \valu + \slope/4) -- (-\maxlevel, \valu) --cycle;

%Drawing the axis (remaining axis)
\draw[black, thick] (0,0) -- (0,\valu + \slope/4);
\end{tikzpicture}
\caption{Example of convex set with vertex at infinity.}
\label{fig:VertexInf}
\end{figure}

We state now an analogue of Theorem \ref{thm:FiniteConv}. Although it can be derived from this result and the fact that the convolution of convex mappings preserves subordination (see Theorem~4.1 in \cite{MR0328051}), we present an alternative and self-contained proof.

\begin{theo} \label{unbounded}
Let $g_1, g_2$ be unbounded convex mappings defined in $\D$ with angles $\pi\al, \pi\beta$ at infinity, respectively.

\sm
\noi
(i) If $\al+\beta<1$ then $g_1*g_2$ is bounded.

\sm
\noi
(ii) If $\al+\beta\geq1$ then $g_1*g_2$ is unbounded and the angle at infinity is at least $\pi(\al+\beta-1)$.

$ $ The angle at infinity is equal to $\pi(\al+\beta-1)$ if both regions $g_1(\D), g_2(\D)$ admit sectors

$ $ of minimal amplitude containing them.
 \end{theo}

\begin{proof}
We may assume that both $\al, \beta <1$. After a rotation, we may assume that both regions contain the ray $(0,\infty)$, with $g_1(1)=g_2(1)=\infty$ and $0\in g_1(\T), g_2(\T)$ as shown in the Figure \ref{fig:Normalization}.

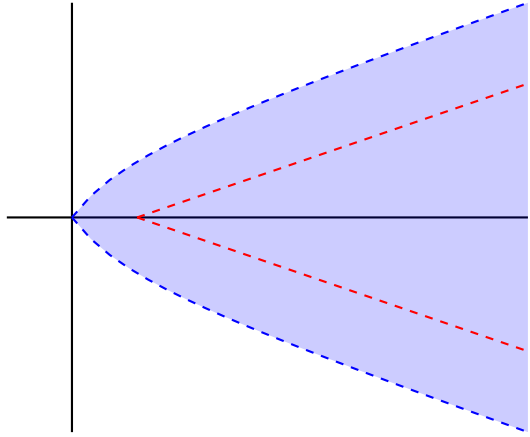
\begin{figure}[h]
\centering
\begin{tikzpicture}

\begin{axis}[axis lines=middle,
            axis line style={-, thick},
            xmin=-1,
            ytick=\empty,
            xtick=\empty]

%Drawing g(D)
\addplot[name path=F,blue, thick, dashed, domain={0:7}] {x/3 + 1 - 1/(x+1)};
\addplot[name path=G,blue, thick, dashed, domain={0:7}] {-x/3 - 1 + 1/(x+1)};
\addplot[color=blue, opacity=0.2]fill between[of=F and G, soft clip={domain=0:7}];

%Drawing the internal angle
\addplot[red, thick, dashed, domain={1:7}] {x/3 - 1/3};
\addplot[red, thick, dashed, domain={1:7}] {-x/3 + 1/3};

\end{axis}

\end{tikzpicture}
\caption{Example of $g_j (\D)$ after the normalization.}
\label{fig:Normalization}
\end{figure}

Using the integral representation of the convolution, we have that
$$
(g_1*g_2)(z)=\frac{1}{2\pi}\int_{-\pi}^{\pi} g_1 \left( \frac{1+r}{2} e^{it} \right)g_2 \left( \frac{2r e^{i\theta}}{1 + r} e^{-it} \right) \, dt, \quad z = r e^{i\theta}.
$$

If $z \ri e^{it_0}\neq 1$ then the integral remains bounded because the factors are each integrable and unbounded at different values of $t$. In other words, $g_1*g_2$ can become unbounded only as $z \ri 1$. Because the convolution is continuous in the spherical metric, we may assume that $z$ is real. For small $\delta>0$ fixed, the integral
$$
\int_{\delta<|t|\leq \pi} g_1 \left(\frac{1+r}{2} e^{it} \right)g_2 \left(\frac{2r}{1+r} e^{-it} \right) \, dt
$$
remains bounded as $r\ri 1$, hence it suffices lo analyze the behavior of
$$
I(r)=\int_{|t|\leq\delta}g_1 \left(\frac{1+r}{2} e^{it} \right)g_2 \left(\frac{2r}{1+r} e^{-it} \right) \, dt
$$
as $r\ri 1$. For this we write $g_1=u_1+iv_1$,  $g_2=u_2+iv_2$, and express
$$
I(r)=\int_{|t|\leq\delta}\left[U(r,t)+iV(r,t)\right] \, dt,
$$
where
$$
U(r,t)=u_1 \left(\frac{1+r}{2} e^{it} \right)u_2\left(\frac{2r}{1+r} e^{-it}\right)-v_1\left(\frac{1 + r}{2} e^{it}\right)v_2\left(\frac{2r}{1+r} e^{-it}\right)
$$
and
$$
V(r,t)=u_1 \left(\frac{1+r}{2} e^{it} \right) v_2 \left(\frac{2r}{1+r} e^{-it} \right)+v_1 \left(\frac{1+r}{2} e^{it}\right)u_2\left(\frac{2r}{1+r} e^{-it} \right).
$$
The normalizations of $g_1, g_2$ imply that $u_1(\frac{1+r}{2}e^{it}), u_2(\frac{2r}{1+r} e^{-it})>0$ while $v_1(\frac{1+r}{2}e^{it})$ and $v_2(\frac{2r}{1+r}e^{-it})$ have the sign of $t$. Furthermore, all four quantities $u_1, u_2, v_1, v_2$ will be comparable in size because the domains contain and are contained in certain sectors.

\sm
For part (i) we use \eqref{eqn:UpperBound} for $g_1$ and $g_2$ with angles $\al'>\al$ and $\beta'>\beta$ and $\al'+\beta'<1$ to conclude that $I(r)$ will stay bounded as $r\ri 1$. For part (ii), we see from the above remarks on the size and signs of the $u_j, v_k$ that it suffices to analyze no only the real part of $I(r)$, but just the integral
$$
J(r)=\int_0^{\delta}u_1\left(\frac{1+r}{2}e^{it}\right)u_2 \left(\frac{2r}{1+r}e^{-it}\right) \, dt
$$
as $r\ri 1$.

Next, a simple calculation shows that
$$
\frac{1}{|1-se^{it}|^{\gamma}}\sim \frac{1}{\left[(1-s)^2+t^2\right]^{\frac{\gamma}{2}}}
$$
as $s\ri 1$, hence by \eqref{eqn:LowerBound} we have
$$
J(r)\gtrsim\int_0^{\delta}\frac{dt}{\left[(1-r)^2+t^2\right]^{\frac{\al+\beta}{2}}}=\frac{1}{(1-r)^{\al+\beta-1}}\int_0^{\delta/(1-r)}\frac{dx}{(1+x^2)^{ \frac{\al+\beta}{2}}} .
$$
This proves that $g_1*g_2$ becomes unbounded as $z\ri 1$ and, because of \eqref{eqn:UpperBound}, that the angle at $(g_1*g_2)(1)=\infty$ cannot be smaller than $\pi(\al+\beta-1)$.
Finally, if both domains admit minimizing sectors containing them, we will also have $J(r)\lesssim(1-r)^{1-\al-\beta}$, showing that the angle at infinity of the convolutions must be $\pi(\al+\beta-1)$.
\end{proof}

\begin{coro}
Let $g_1, \ldots, g_n$ be unbounded convex mappings with angles at infinity $\pi\al_1, \ldots, \pi\al_n$. Then the probability that $g_1*\cdots*g_n$ is unbounded is equal to $1/n!$.
\end{coro}

\begin{proof}
Each stage of the successive convolutions must be unbounded, otherwise the final convolution will be bounded. Therefore, $\al_1+\al_2\geq1$, then $\al_1+\al_2+\al_3\geq2$, and so on.  We conclude that $\al_1+\cdots+\al_n\geq n-1$. A standard calculation shows that the volume of this region in the $n$-cube $[0,1]^n$ is equal to $1/n!$, as claimed.

\end{proof}

\bibliographystyle{amsplain} \bibliography{Biblio}
\end{document}